\documentclass[a4paper,10pt]{amsart}

\usepackage[dvipsnames]{xcolor}
\usepackage[utf8]{inputenc}
\usepackage{graphicx}
\usepackage{amsmath}
\usepackage{amssymb}
\usepackage{amsthm}
\usepackage{url}
\usepackage{mathtools}

\usepackage{tikz}
\usepackage{pst-node}
\usepackage[all]{xy}
\usepackage{array}
\usepackage{a4wide}

\newtheorem*{defi-intro}{Definition}
\newtheorem*{Theo}{Theorem}

\newtheorem*{TheoremBruzzoGrassi}{Theorem 4.11 of \cite{ugo_grassi}}

\newtheorem{Theorem}{Theorem}
\newtheorem{Corollary}[Theorem]{Corollary}
\newtheorem{Proposition}[Theorem]{Proposition}

\theoremstyle{definition}

\theoremstyle{definition}

\theoremstyle{remark}
\newtheorem{remark}[Theorem]{Remark}

\theoremstyle{remark}

\title[On the codimension of Noether-Lefschetz loci for toric threefolds]{On the codimension of Noether-Lefschetz loci\\for toric threefolds}
\author{Valeriano Lanza and Ivan Martino}
\date{\today}

\begin{document}
\begin{abstract}
In this manuscript we sharpen the lower bound on the codimension of the irreducible components of the Noether-Lefschetz locus of surfaces in projective toric threefolds given in \cite{ugo_grassi}.
We also provide a simpler proof of Theorem 4.11 in \cite{ugo_grassi}, which allows one to avoid some technical assumptions.
\end{abstract}
\maketitle

Let $\mathbb{P}_{\Sigma}$ be a projective toric threefold with orbifold singularities, $\beta$ a nef class in the class group $\operatorname{Cl}(\mathbb{P}_{\Sigma})$, and $\mathcal M_{\beta}$ the moduli space of surfaces in $\mathbb{P}_{\Sigma}$ of degree $\beta$ modulo automorphisms of $\mathbb{P}_{\Sigma}$.
The Noether-Lefschetz locus $U_{\beta}$ with respect to $\beta$ is the subscheme of $\mathcal M_{\beta}$ corresponding to quasi-smooth surfaces whose Picard number is strictly larger than the one of $\mathbb{P}_{\Sigma}$.

Let $\eta$ be a primitive ample Cartier class and suppose that $\beta$ is a Cartier divisor of the form $-K_{\mathbb{P}_{\Sigma}} + n \eta$, where $K_{\mathbb{P}_{\Sigma}}$ is the canonical divisor of $\mathbb{P}_{\Sigma}$ and $n$ is a non-negative integer; in this specific case, we denote $U_{\beta}$ by $U_{\eta}(n)$. Note that, to have such a $\beta$, one has to implicitly assume that $K_{\mathbb{P}_{\Sigma}}$ is Cartier, which forces $\mathbb{P}_{\Sigma}$ to be Gorenstein.
%Therefore, the latter assumption will be understood throughout the paper even where not explicitly mentioned.

Assume the multiplication map
\begin{equation*}%\label{eq:surjection}\tag{$\star$}
		\operatorname{H}^0(\mathbb{P}_{\Sigma},\mathcal{O}_{\mathbb{P}_{\Sigma}}(\beta))\otimes \operatorname{H}^0(\mathbb{P}_{\Sigma},\mathcal{O}_{\mathbb{P}_{\Sigma}}(n\eta))\rightarrow \operatorname{H}^0(\mathbb{P}_{\Sigma},\mathcal{O}_{\mathbb{P}_{\Sigma}}(\beta+n\eta))                                                                                                                           
\end{equation*}
to be surjective; we shall refer to this as ``condition $(\star)$.'' Hence, the very general quasi-smooth surface in the linear system defined by $\beta$ has the same Picard number as $\mathbb{P}_{\Sigma}$, see Theorem 3.5 of \cite{BruzzoGrassi-Picard-group-hypersurfaces}; in other words $U_{\eta}(n)$ is a countable union of closed subschemes of positive codimension. Bruzzo and Grassi prove the following:

\begin{TheoremBruzzoGrassi} 
%Let $\mathbb{P}_{\Sigma}$, $\eta$, $\beta$, $n$ as above. 
Assume that the multiplication maps
\[
 \operatorname{H}^0(\mathbb{P}_{\Sigma},\mathcal{O}_{\mathbb{P}_{\Sigma}}(\beta))\otimes \operatorname{H}^0(\mathbb{P}_{\Sigma},\mathcal{O}_{\mathbb{P}_{\Sigma}}(k\eta))\rightarrow \operatorname{H}^0(\mathbb{P}_{\Sigma},\mathcal{O}_{\mathbb{P}_{\Sigma}}(\beta+k\eta))                                                                                                                           
\]
are surjective for all positive integer $k$ (in particular condition $(\star)$ is fulfilled). Assume also the vanishings
\[
 \operatorname{H}^1(\mathbb{P}_{\Sigma},\mathcal{O}_{\mathbb{P}_{\Sigma}}(\beta-\eta))=\operatorname{H}^2(\mathbb{P}_{\Sigma},\mathcal{O}_{\mathbb{P}_{\Sigma}}(\beta-\eta))=\operatorname{H}^2(\mathbb{P}_{\Sigma},\mathcal{O}_{\mathbb{P}_{\Sigma}}(\beta-2\eta))=0\,.
\]

Then, for any irreducible component $U$ of $U_{\eta}(n)$, 
	\begin{center}
	\begin{tabular}{crcl}
	i) & $\eta \ \ (-1)-\mbox{regular}$ &$\Rightarrow$ & $\operatorname{codim}U\geq n+1$;\\
	ii) & $\eta \ \ 0-\mbox{regular}$ &$\Rightarrow$ & $\operatorname{codim}U \geq n$.
	\end{tabular}
	\end{center}
\end{TheoremBruzzoGrassi}

Nevertheless, they provide many examples where the chosen ample class $\eta$ is $0$-regular and the locus $U_{\eta}(n)$ has a component of codimension precisely $n+1$, see Section 5 in \cite{ugo_grassi}. 
These are the natural generalizations of the examples that achieve the \emph{minimal} bound in the classical setting; namely, the surfaces in $\mathbb{P}^3$ that contain a line \cite{Voisin-Composantes, Green-KoszulCohomology, Green-A-new-proof-of-the-explicit, Green-Components}.
For this reason one is led to believe that the bound in \textit{ii)} may not be sharp: we show here that this is precisely the case.

\begin{Theo}
	Let $\mathbb{P}_\Sigma$ be a Gorenstein projective toric threefold with orbifold singularities, $\eta$ a primitive ample Cartier class which is $0$-regular, and $\beta$ a Cartier divisor of the form $-K_{\mathbb{P}_{\Sigma}} + n \eta$, where $K_{\mathbb{P}_{\Sigma}}$ is the canonical divisor of $\mathbb{P}_{\Sigma}$ and $n$ is a non-negative integer.
	
	Assume condition $(\star)$.
	 If $\operatorname{H}^1(\mathbb{P}_{\Sigma}, \mathcal{O}_{\mathbb{P}_{\Sigma}}(\beta-\eta))=\operatorname{H}^2(\mathbb{P}_{\Sigma}, \mathcal{O}_{\mathbb{P}_{\Sigma}}(\beta-2\eta))=0$, then \[\operatorname{codim}U\geq n+1\] for every irreducible component $U$ of $U_{\eta}(n)$.
\end{Theo}

\begin{remark}
 The bounds in Theorem 4.12, Corollary 4.13, and Theorem 4.15 of \cite{ugo_grassi} can be accordingly upgraded. Notice that the hypoteses of each of the results just mentioned imply those of Theorem 4.11 of \cite{ugo_grassi}, and so, a fortiori, ours.
\end{remark}

The new bound is sharp for the following threefolds and suitable choices of $\eta$, see Section 5 of \cite{ugo_grassi}:
\begin{itemize}
	\item toric simplicial Gorenstein threefolds with nef anticanonical bundle;
	\item the projective space blown-up along a line, $\hat{\mathbb{P}}^3$; 
	\item $\mathbb{P}^1\times \mathbb{P}^2$;
	\item a small resolution of the cone over a quartic surface in $\mathbb{P}^3$;
	\item the weighted projective space $\mathbb{P}[1,1,2,2]$.
\end{itemize}

\bigskip
\begin{center}
 $\ast$~$\ast$~$\ast$
\end{center}
\medskip

Before proving the theorem, we first notice that, for what concerns Theorem 4.11 of \cite{ugo_grassi}, the discrepancy between the case where $\eta$ is $(-1)$-regular and the one where it is just $0$-regular has its origin in Proposition 3.5 of \cite{ugo_grassi}, which inspects the regularity (with respect to $\eta$) of the tensor powers of a $1$-regular locally free sheaf $\mathcal{F}$ on $\mathbb{P}_{\Sigma}$.

Our idea is to use instead a natural generalization of Corollary 1.8.10 of \cite{laza1}.

\begin{Proposition}\label{Prop:CorLaza}
Let $X$ be a projective variety together with an ample line bundle $B$ which is globally generated and $0$-regular.
If $\mathcal{F}$ is an $m$-regular locally free sheaf on $X$, then the $p$-fold tensor power $\mathcal{F}^{\otimes p}$ is $(pm)$-regular. In particular, $\Lambda^p\mathcal{F}$ and $S^p\mathcal{F}$ are likewise $(pm)$-regular.
%\end{CorLaza}
\end{Proposition}
\begin{proof}
The proof of Corollary 1.8.10 of \cite{laza1}, which is stated for locally free sheaves on a projective space $\mathbb{P}$ 
which are regular with respect to $\mathcal{O}_{\mathbb{P}}(1)$, also works in this more general context. The $0$-regularity of $B$ ensures the existence of a linear resolution for $\mathcal{F}$ of the form
\[
 \dots\rightarrow\bigoplus B^{-m-2}\rightarrow\bigoplus B^{-m-1}\rightarrow\bigoplus B^{-m}\rightarrow\mathcal{F}\rightarrow 0\,,
\]
cf. Proposition 1.8.8 of \cite{laza1} and Corollary 3.3 of \cite{ugo_grassi}.
\end{proof}

It is well-known that ample line bundles are always globally generated on projective toric varieties.

\subsection*{Proof of the Theorem}
What follows logically corresponds to \S4.1 of \cite{ugo_grassi}. We pick a base-point-free linear system $W$ in $\operatorname{H}^0(\mathbb{P}_\Sigma,\mathcal{O}_{\mathbb{P}_\Sigma}(\beta))$ 
and we select a complete flag of linear subspaces
\[
	W=W_c\subset W_{c-1}\subset \dots \subset W_1\subset W_0=\operatorname{H}^0(\mathbb{P}_\Sigma,\mathcal{O}_{\mathbb{P}_\Sigma}(\beta))\,.
\]
This means that the codimension of $W_i$ is $i$ (and in particular $c$ is the codimension of $W$).
We define the vector bundle $M_i$ over $\mathbb{P}_\Sigma$ as the kernel of the natural (surjective) map $W_i \otimes \mathcal{O}_{\mathbb{P}_\Sigma}\rightarrow\mathcal{O}_{\mathbb{P}_\Sigma}(\beta)$.

\begin{Proposition}
	$M_0$ is $1$-regular with respect to $\eta$.
\end{Proposition}
\begin{proof}
	By definition, we would like to show that $\operatorname{H}^q(\mathbb{P}_\Sigma, M_0((1-q)\eta))=0$ for all positive $q$.
	Consider the long exact sequence in cohomology associated with
	\begin{equation}\label{eq:def-M0}\tag{$\dagger$}
		0\rightarrow M_0 \rightarrow W_0 \otimes \mathcal{O}_{\mathbb{P}_\Sigma}\rightarrow \mathcal{O}_{\mathbb{P}_\Sigma}(\beta) \rightarrow 0\,,
	\end{equation}		
	that is
	\[
	0\rightarrow \operatorname{H}^0(\mathbb{P}_\Sigma,M_0) \rightarrow \operatorname{H}^0(\mathbb{P}_\Sigma,W_0 \otimes \mathcal{O}_{\mathbb{P}_\Sigma}) \stackrel{\pi}{\rightarrow} \operatorname{H}^0(\mathbb{P}_\Sigma,\mathcal{O}_{\mathbb{P}_\Sigma}(\beta)) \rightarrow \operatorname{H}^1(\mathbb{P}_\Sigma,M_0) \rightarrow \cdots
	\]
	Since $\pi$ is surjective, $\operatorname{H}^1(\mathbb{P}_\Sigma,M_0)=0$.
	
	The vanishing of $\operatorname{H}^2(\mathbb{P}_\Sigma,M_0(-\eta))$ is obtained by tensoring $\eqref{eq:def-M0}$ by $\mathcal{O}_{\mathbb{P}_\Sigma}(-\eta)$ and considering that $\operatorname{H}^2(\mathbb{P}_\Sigma,M_0(-\eta))$ lies between two zeros in the long exact sequence
	\[
	\cdots \rightarrow \operatorname{H}^1(\mathbb{P}_\Sigma,\mathcal{O}_{\mathbb{P}_\Sigma}(\beta-\eta)) \rightarrow \operatorname{H}^2(\mathbb{P}_\Sigma,M_0(-\eta)) \rightarrow \operatorname{H}^2(\mathbb{P}_\Sigma,W_0\otimes\mathcal{O}_{\mathbb{P}_\Sigma}(-\eta)) \rightarrow \cdots
	\]
	Note that $\operatorname{H}^2(\mathbb{P}_\Sigma,W_0\otimes\mathcal{O}_{\mathbb{P}_\Sigma}(-\eta))=0$ because $\eta$ is $0$-regular with respect to itself.
	
	One argues similarly for $\operatorname{H}^3(\mathbb{P}_\Sigma,M_0(-2\eta))$. 
\end{proof}

\begin{Corollary}\label{cor:mah}
	%Let $\eta$ be a $0$-regular ample line bundle.
	%
	For all $i=0, \dots, c$, $\operatorname{H}^q(\mathbb{P}_\Sigma,\wedge^p M_i(n\eta))=0$ if $q\geq 1$ and $n+q\geq p+i$.%\footnote{\green This corresponds to Corollary 4.9 of \cite{ugo_grassi}.}
\end{Corollary}
\begin{proof}
	As a straightforward consequence of Mumford's Theorem \cite[Theorem 1.8.5]{laza1} one has that a coherent sheaf $\mathcal{F}$ on $\mathbb{P}_\Sigma$ is $m$-regular with respect to $\eta$ if and only if
 \[
  H^{q}(\mathbb{P}_\Sigma, \mathcal{F}\otimes\mathcal{O}_{\mathbb{P}_\Sigma}(n\eta))=0\,,
 \]
for all integers $q>0$, $n\geq m-q$.
	%By virtue of the previous remark, 
	Hence, for $i=0$ the statement is equivalent to $\wedge^p M_0$ being $p$-regular with respect to $\eta$ and this is proved in Proposition \ref{Prop:CorLaza}. The proof then follows as for Lemma 2 in \cite{Green-A-new-proof-of-the-explicit}.
\end{proof}

\begin{Proposition}\label{prop-c-codim}
	If $c=\operatorname{codim}W\leq n$, then the map $W \otimes \operatorname{H}^0(\mathbb{P}_\Sigma,\mathcal{O}_{\mathbb{P}_\Sigma}(n\eta)) \rightarrow\operatorname{H}^0(\mathcal{O}_{\mathbb{P}_\Sigma}(\beta+n\eta))$ is surjective.%\footnote{\green This corresponds to Proposition 4.10 of \cite{ugo_grassi}.}
\end{Proposition}
\begin{proof}
	We consider the short exact sequence
	\[
		0\rightarrow M_c \rightarrow W \otimes \mathcal{O}_{\mathbb{P}_\Sigma}\rightarrow\mathcal{O}_{\mathbb{P}_\Sigma}(\beta) \rightarrow 0\,,
	\]
	and twist it by $\mathcal{O}_{\mathbb{P}_\Sigma}(n\eta)$. We pass to the long exact sequence in cohomology:
	\[
		\dots \rightarrow \operatorname{H}^0(\mathbb{P}_\Sigma,W \otimes\mathcal{O}_{\mathbb{P}_\Sigma}(n\eta))\rightarrow \operatorname{H}^0(\mathbb{P}_\Sigma,\mathcal{O}_{\mathbb{P}_\Sigma}(\beta+n\eta)) \rightarrow \operatorname{H}^1(\mathbb{P}_\Sigma,M_c(n\eta))\rightarrow \dots
	\]
	By applying Corollary \ref{cor:mah} with $p=q=1$, one gets that $\operatorname{H}^1(\mathbb{P}_\Sigma,M_c(-n\eta))$ is zero.
\end{proof}

Finally, the proof of the Theorem follows precisely as the one of Theorem 4.11 of \cite{ugo_grassi} with the only difference that we use the previous proposition instead of Proposition 4.10 of \cite{ugo_grassi}.

\subsection*{Acknowledgments}
This work has been inspired by a talk given by Antonella Grassi at the Special Session on 
\emph{Topics in Toric Geometry} in the AMS Sectional Meeting at Northeastern University. 
In subsequent chat with the second author and Emanuele Ventura, she conjectured the improvement of the bound. We thank her for that cue and, together with Ugo Bruzzo, also for the interesting correspondence that ensued.

%\subsection*{Supports}
The first author is supported by the FAPESP postdoctoral grant no.2017/22052-9 (Bolsa Estágio de Pesquisa no Exterior). The research this paper is based upon was made while he was visiting the Department of Mathematics at Northeastern University (Boston), and he wishes to thank that institution for hospitality.
The second author has been partially supported by the \emph{Zelevinsky Research Instructor Fund} and he is currently supported by the \emph{Knut and Alice Wallenberg Fundation} and by the \emph{Royal Swedish Academy of Science}.

\bibliographystyle{alpha}
\def\cprime{$'$} \def\cprime{$'$} \def\cprime{$'$} \def\cprime{$'$}

\vspace{0.5cm}

\noindent
 {\scshape Valeriano Lanza}\\
 {\scshape IMECC -- UNICAMP\\Rua Sérgio Buarque de Holanda, 651\\ 13083-859 Campinas, SP, Brazil}\\
 {\itshape E-mail address}: \texttt{valeriano@ime.unicamp.br}

\vspace{0.5cm}

\noindent
 {\scshape Ivan Martino}\\
 {\scshape Department of Mathematics, Northeastern University,\\ Boston, MA 02115, USA}.\\
 {\itshape E-mail address}: \texttt{i.martino@northeastern.edu}
\end{document}